\documentclass[11pt, leqno]{article}
\usepackage{amsmath,amssymb,amsthm, amsfonts, epsfig}

\numberwithin{equation}{section}

\usepackage{hyperref}

\usepackage{amsmath}

\usepackage{amssymb}

\usepackage{color}

\usepackage{ulem}

\usepackage{epsfig}

\usepackage[mathscr]{eucal}

\usepackage{stackrel}

\usepackage{mathptmx} 
\usepackage[scaled=0.90]{helvet} 
\usepackage{courier} 
\normalfont
\usepackage[T1]{fontenc}

\title{Non-existence of dead cores in fully nonlinear elliptic models}
\author{\it by \smallskip \\
 Jo\~{a}o V{i}tor da Silva\footnote{\noindent Departamento de Matem\'atica - Instituto de Ci\^{e}ncias Exatas - Universidade de Bras\'{i}lia. Campus Universit\'{a}rio Darcy Ribeiro, 70910-900, Bras\'{i}lia - Distrito Federal - Brazil. \noindent \texttt{E-mail address: J.V.Silva@mat.unb.br}} \footnote{\noindent Instituto de Investigaciones Matem\'{a}ticas Luis A. Santal\'{o} (IMAS) - CONICET (Argentine), Ciudad Universitaria, Pabell\'{o}n I (1428) Av. Cantilo s/n - Buenos Aires \texttt{E-mail address: jdasilva@dm.uba.ar}}, \,\,\, Disson dos Prazeres\footnote{\noindent Universidade Federal de Sergipe, Departmento de Matem\'{a}tica - Aracaju, Brazil. \noindent \texttt{E-mail address: disson@mat.ufs.br}}  \,\,\,\,$\&$ \,\,\,\, Humberto Ramos Quoirin\footnote{CIEM-Famaf, Universidad Nacional de C\'{o}rdoba, (5000)
C\'{o}rdoba, Argentina \noindent \texttt{E-mail address: humbertorq@gmail.com}}
}
\date{}

\newlength{\hchng}
\newlength{\vchng}
\def \R {\mathbb{R}}

\def \tr {\mathrm{Tr}}

\newcommand{\defeq}{\mathrel{\mathop:}=}

\newtheorem{Proof of Theorem a small}[]{Proof of Theorem (\ref{main result 1}) }

\newtheorem{theorem}{Theorem}[section]

\newtheorem{lemma}[theorem]{Lemma}

\newtheorem{corollary}[theorem]{Corollary}

\theoremstyle{definition}

\newtheorem{definition}[theorem]{Definition}

\newtheorem{example}[theorem]{Example}

\theoremstyle{remark}

\newtheorem{remark}[theorem]{Remark}

\numberwithin{equation}{section}

\newcommand{\intav}[1]{\mathchoice {\mathop{\vrule width 6pt height 3 pt depth  -2.5pt
\kern -8pt \intop}\nolimits_{\kern -6pt#1}} {\mathop{\vrule width
5pt height 3  pt depth -2.6pt \kern -6pt \intop}\nolimits_{#1}}
{\mathop{\vrule width 5pt height 3 pt depth -2.6pt \kern -6pt
\intop}\nolimits_{#1}} {\mathop{\vrule width 5pt height 3 pt depth
-2.6pt \kern -6pt \intop}\nolimits_{#1}}}

\begin{document}
\maketitle

\begin{abstract}
We investigate non-existence of nonnegative dead-core solutions for the problem $$|Du|^\gamma F(x, D^2u)+a(x)u^q = 0 \quad \mbox{in} \quad \Omega, \quad u=0 \quad \mbox{ on } \quad \partial\Omega.$$
Here $\Omega \subset \mathbb{R}^N$ is a bounded smooth domain, $F$ is a fully nonlinear elliptic operator, $a: \Omega \to \mathbb{R}$ is a sign-changing weight, $\gamma \geq 0$, and $0<q<\gamma+1$.
We show that this problem has no non-trivial dead core solutions if either $q$ is close enough to $\gamma+1$ or the negative part of $a$ is sufficiently small. In addition, we obtain the existence and uniqueness of a positive solution under these conditions on $q$ and $a$. Our results extend previous ones established in the semilinear case, and are new even for the simple model
$|D u(x)|^{\gamma} \mathrm{Tr}(\mathrm{A}(x) D^2 u(x)) + a(x)u^{q}(x) = 0$,
where $\mathrm{A} \in C^0(\Omega;Sym(N))$ is a uniformly elliptic and non-negative matrix.
\end{abstract}

\section{Introduction}

Let $\Omega \subset \R^N$ be a bounded and smooth domain with $N \geq 1$. This work is devoted to the problem
\begin{equation}\tag{{\bf $\mathrm{P}_{a,q}$}}\label{MEq}
\left\{
\begin{array}{rcrcl}
|Du|^\gamma F(x, D^2u)+a(x)u^q& =& 0 & \mathrm{in} & \Omega\\
u &\geq& 0 & \mathrm{in} & \Omega\\
u& = & 0 & \mathrm{on} & \partial\Omega,
\end{array}
\right.
\end{equation}
where $\gamma \geq 0$, $q \in (0,\gamma+1)$, $a \in C(\overline{\Omega})$ changes sign, and $F: \Omega \times \mathrm{Sym}(N) \to \R$ is a second order fully nonlinear elliptic operator satisfying the following assumptions:

\begin{enumerate}
\item[{\bf(F1)}]\label{F1}[{\bf $(\lambda, \Lambda)$-Ellipticity condition}] There exist constants $ \Lambda \geq \lambda >0$ (known as \textit{ellipticity parameters}) such that for any $x \in \Omega$ and $X, Y \in\mathrm{Sym}(N) $, with $Y\ge 0$ there holds
$$
   \mathcal{M}_{\lambda, \Lambda}^{-}(Y)\leq F(x, X+Y)-F(x, X)\leq \mathcal{M}_{\lambda, \Lambda}^{+}(Y),
$$
where $\mathcal{M}^{\pm}_{\lambda,\Lambda}$ denote the \textit{Pucci's extremal operators}:
$$
     \mathcal{M}^{+}_{\lambda,\Lambda}(X) := \lambda \cdot \sum_{e_i <0} e_i(X)  + \Lambda \cdot \sum_{e_i >0} e_i(X) \quad \mbox{and} \quad \mathcal{M}^{-}_{\lambda,\Lambda}(X) := \lambda \cdot \sum_{e_i >0} e_i(X)  + \Lambda \cdot \sum_{e_i <0} e_i(X)
$$
and $\{e_i(X) : 1 \le i \le N\}$ are the eigenvalues of $X$.

\item[{\bf(F2)}]\label{F2}[{\bf $1$-Homogeneity condition}] For all $s \in \R^{\ast}$, $x \in \Omega$ and $X \in Sym(N)$ there holds
$$
  F(x, sX) = |s|F(x, X).
$$
\item[{\bf(F3)}]\label{F3}[{\bf Continuity condition}] There exists a modulus of continuity $\omega : [0, \infty) \to [0, \infty)$ with $\omega(0) = 0$ such that for all $(x, y,X) \in \Omega \times \Omega \times Sym(N)$
$$
      |F(x,X)-F(y,X)| \leq \omega(|x-y|)\|X\|.
$$

\end{enumerate}

Some examples of operators satisfying our assumptions are given next:
\begin{enumerate}
  \item {\bf Second order operators in non-divergence form:}
  $$
  F(x, D^2 u)  := \mathcal{L} u(x) = \sum_{i, j=1}^{N} a_{ij}(x)\partial_{ij} u(x),
$$
where $(a_{ij})_{i,j=1}^N$ is a uniformly elliptic and symmetric matrix with (uniformly) continuous coefficients.

  \item {\bf Hamilton-Jacobi-Bellman and Isaac's type operators:}

 The following Hamilton-Jacobi-Bellman (HJB in short) and Isaac's type operators, which appear in control problems for stochastic systems and in the theory of (zero-sum two-player) stochastic differential games, see \textit{e.g} \cite{BL, Lions1, Lions2, Lions3}:
 $$
      \displaystyle F(x, D^2 u) \defeq  \inf_{\hat{\alpha} \in \mathcal{A}}\left(\mathcal{L}_{\hat{\alpha}} u(x)\right) \quad \left(\text{respect.}\,\,\, \sup_{\hat{\alpha} \in \mathcal{A}} \left(\mathcal{L}_{\hat{\alpha}} u(x)\right)\right),
 $$
and
$$
      \displaystyle F(x, D^2 u) \defeq \sup_{\hat{\beta} \in \mathcal{B}} \inf_{\hat{\alpha} \in \mathcal{A}}\left(\mathcal{L}_{\hat{\alpha}\hat{\beta}} u(x)\right) \quad \left(\text{respect.}\,\,\,\inf_{\hat{\beta} \in \mathcal{B}} \sup_{\hat{\alpha} \in \mathcal{A}} \left(\mathcal{L}_{\hat{\alpha}\hat{\beta}} u(x)\right)\right).
$$
Here $\displaystyle \mathcal{L}_{\hat{\alpha}\hat{\beta}} u(x) = \sum_{i, j=1}^{N} a_{ij}^{\hat{\alpha}\hat{\beta}}(x)\partial_{ij} u(x)$ is a family of uniformly elliptic operators with coefficients enjoying a universal modulus of continuity and ellipticity constants $\Lambda\geq \lambda>0$.

  \item {\bf Pucci's extremal operators:}
$$
      \displaystyle \mathcal{M}_{\lambda, \Lambda}^{+}(X)  = \max\left\{\tr(AX): \lambda I\leq A \leq \Lambda I \quad \text{for}\quad A \in Sym(N)\right\}
$$
$$
\displaystyle \mathcal{M}_{\lambda, \Lambda}^{-}(X)  = \min\left\{\tr(AX): \lambda I\leq A \leq \Lambda I \quad \text{for}\quad A \in Sym(N)\right\}
$$

\item {\bf $p-$Laplacian type operators:}

Other interesting examples are given by the $p-$Laplacian operator (in non-divergence form) with $\gamma = p-2$ (for $p > 1$)
     $$
     G_p(\xi, X) = |\xi|^{\gamma}F_p(\xi, X)
     $$
where
$$
\begin{array}{ccl}
  F_p(\xi, X) & \defeq & \left(\tr(X)+(p-2)|\xi|^{-2}\left\langle X\xi, \xi\right\rangle\right) \\
   & = & \tr\left[\left(I+(p-2)\frac{\xi\otimes \xi}{|\xi|^2}\right)X\right].
\end{array}
$$
Furthermore, we can also quote non-variational generalizations of the $p-$Laplacian operator, with continuous coefficients, as follows:
$$
\begin{array}{ccl}
     G_{c_0}(x, \xi, X) & \defeq &  \tr(\mathcal{A}_1(x)X)-c_0|\xi|^{-2}\left\langle X\mathcal{A}_2(x)\xi \mathcal{A}_2(x) \xi\right\rangle\\
     &  = & \tr\left[\left(\mathcal{A}_1(x)+c_0\frac{\mathcal{A}_2(x)\xi\otimes \mathcal{A}_2(x)\xi}{|\xi|^2}\right)X\right].
\end{array}
$$
 where $\mathcal{A}_1, \mathcal{A}_2: \Omega \to \text{Sym}(N)$ are $\theta-$H\"{o}lderian functions with $\theta> \frac{1}{2}$, $\lambda I \leq \mathcal{A}_1 \leq  \Lambda I$,$- \sqrt{\lambda} I \leq \mathcal{A}_2 \leq \sqrt{\lambda} I$ and $c_0>-1$.

We refer the reader to \cite{ART15}, \cite{BD07}, \cite{BD09}, \cite{BD2}, \cite[Chapter 2]{CC95}, \cite{Trud82} and  references therein for a number of examples of fully nonlinear operators with such structural properties.
\end{enumerate}

Note that when $\gamma>0$ the operator $|Du|^\gamma F(x, D^2u)$ is degenerate in the gradient. Thanks to $(F1)$ this operator is $(\gamma+1)$-homogeneous and since $q<\gamma+1$, the reaction term in the equation is 'sub-homogeneous' in comparison with the operator. This fact, combined with the change of sign of $a$, prevents the use of the strong maximum principle and enables the existence of {\it dead core} solutions, i.e. nontrivial solutions vanishing in some part of the domain. It is worth pointing out that the condition $a^+ \not \equiv 0$ is necessary for the existence of nontrivial solutions of \eqref{MEq}. We refer to \cite{BdaL} for an essay on the validity of the strong maximum principle and \cite{BV, Diaz, DiazHern, DiazHerr, DiazVer} for several features on dead core type problems. 


Specifically, our main aim in this manuscript is to extend the positivity results established in \cite{KRQU}, where a simplified semilinear version of  \eqref{MEq} has been considered for $\Delta$, the Laplace operator. Based on a continuity argument, it was proved in  \cite[Theorems 1.1 and 1.3]{KRQU} that the positivity property obtained via the strong maximum principle still holds for nonnegative solutions of the equation
$$
  \Delta u+a(x)u^q=0 \quad \text{in} \quad \Omega
$$
as long as $a$ and $q$ nearly satisfy the conditions to apply this principle (i.e. $a \geq 0$ or $q=1$). Such continuity argument combines some ingredients that turn out to be available for \eqref{MEq} as well, namely: $C^{1,\alpha}$ regularity estimates up to the boundary, the sub-supersolution method, the strong maximum principle (with $q=\gamma+1$ or $a \geq 0$), and the homogeneity of the reaction term and the operator.

In view of the non-divergence nature of our equations, we are not allowed to employ the classical theory of weak solutions in Sobolev spaces. Solutions of \eqref{MEq} will then be considered in the viscosity sense, as follows:

\begin{definition}[{\bf Viscosity solutions}] Let  $G: \Omega \times \R \times \R^N \times \text{Sym}(N)$ be  continuous. Then $u \in C(\Omega)$ is a viscosity super-solution (respect. sub-solution) of
\begin{equation}\label{EqDiricProb}
G(x, u, Du, D^2 u) = 0 \quad \mbox{in} \quad \Omega
\end{equation}
if for every $x_0 \in \Omega$ and  $\phi \in C^2(\Omega)$ such that $u-\phi$ has a local minimum at $x_0$ (resp. local maximum at $x_0$), one has
      $$
      G(x_0, u(x_0), D\phi(x_0), D^2 \phi(x_0)) \leq  0 \quad (\mbox{respect.} \,\, \geq 0).
      $$
Finally, $u$ is said to be a viscosity solution of the above equation if it is simultaneously a viscosity super-solution and a viscosity sub-solution.
\end{definition}

A solution $u$ of \eqref{MEq} is said to be:
\begin{itemize}
\item non-trivial if $u \not \equiv 0$;
\item positive if $u>0$ in $\Omega$.
\item a {\it dead core} solution if it vanishes in an open subset of $\Omega$.
\end{itemize}

We state now our main results.
Our first result, proved by the method of sub-solution and super-solution, guarantees that \eqref{MEq} has a nontrivial solution.

\begin{theorem}[{\bf Existence of nontrivial solutions}]\label{existencia}
Assume $(F1)$ and $(F2)$. Then \eqref{MEq} has a nontrivial viscosity solution.
\end{theorem}

 By adapting some ideas from \cite{CP09} and \cite{IS}, we are able to ensure that \eqref{MEq} has at most one positive solution.

\begin{theorem}[{\bf Uniqueness of positive solution}]\label{unicidade}
Assume $(F1)$ and $(F2)$. Then \eqref{MEq} has at most one positive solution.
\end{theorem}

Let us set $\Omega^+ \defeq \{x \in \Omega: a(x)>0\}$ and
\[
\mathcal{P}^{\circ}(\Omega)\defeq \left\{  u\in C_{0}^{1}(\overline{\Omega
}):u>0 \,\,\, \text{ in } \,\,\, \Omega,  \text{ \ and } \,\,\, \frac{\partial u}{\partial {\nu}} <0 \,\,\,\text{
on }\partial\Omega\right\},
\]
which is the interior of the \textit{positivity cone} of $C_{0}^{1}(\overline{\Omega
})$, where $\frac{\partial u}{\partial {\nu}}$ denotes the exterior normal derivative.

The next theorem states that \eqref{MEq} has no non-trivial dead-core solution if either $\|a^-\|_\infty$ is sufficiently small or $q$ is close enough to $\gamma +1$.

\begin{theorem}[{\bf Non-existence of dead-core solutions}]\label{main result 1}
Assume $(F1)$, $(F2)$, $(F3)$, and that $\Omega_+$ has finitely many connected components. Then:
\begin{enumerate}
\item There exists $\delta = \delta (N, \lambda, \Lambda, \gamma, q, a^{+}, \Omega)>0$ such that every nontrivial  solution of \eqref{MEq} belongs to $\mathcal{P}^{\circ}(\Omega)$ if $\|a^{-}\|_{\infty}<\delta$.
\item There exists $q_0=q_0(N, a, \gamma, \lambda, \Lambda, \Omega)\in [0,\gamma + 1)$ such that every nontrivial  solution of \eqref{MEq} belongs to $\mathcal{P}^{\circ}(\Omega)$ if $q\in(q_0, \gamma + 1]$.
\end{enumerate}
 \end{theorem}

Thanks to the previous results, we deduce the existence and uniqueness of a positive solution under the aforementioned conditions on $a$ and $q$.

\begin{corollary}[{\bf Existence and uniqueness of nontrivial solution}]
Under the assumptions of Theorem  \ref{main result 1} the problem \eqref{MEq} has a unique nontrivial solution (which belongs to $\mathcal{P}^{\circ}(\Omega)$) if either $\|a^{-}\|_{\infty}<\delta$ or $q\in(q_0, \gamma + 1)$.
\end{corollary}

To the best of our knowledge and as far as degenerate elliptic models in non-divergence form are concerned, Theorems \ref{existencia}, \ref{unicidade}, and \ref{main result 1} are a novelty even for the model $$
\left\{
\begin{array}[c]{rclcl}
  |Du|^{\gamma}\tr(\mathrm{A}(x)D^2 u) +a(x)u(x)^q& = & 0 & \text{in} & \Omega\\
  u &\geq& 0 & \mathrm{in} & \Omega\\
  u& = & 0 & \text{on} & \partial \Omega
\end{array}
\right.
$$
with $\mathrm{A} \in C^0(\Omega;Sym(N))$ a uniformly elliptic and non-negative matrix.

Our paper is organized as follows: in Section \ref{Prelim} we recall some well-known results concerning the operators considered here. Such results help us in the proofs of our main results. In section \ref{SecExist} we prove the existence theorem and in section \ref{SecUnique} we prove the uniqueness of positive solutions. Finally, in section \ref{SecProofMR} we prove of our main result  and give an example of a problem with a dead-core solution.

\section{Preliminaries}\label{Prelim}

\hspace{0.6cm}Throughout this section we recall the results we shall rely on.
Let us start with a stability result, which claims that the limit of a sequence of viscosity solutions turns out to be a viscosity solution of the limiting equation, whose proof can be found in \cite[Corollary 4.3]{BD3} and \cite[Proposition 2.1]{Ishii89} .

\begin{lemma}[{\bf Stability}] Let $G_k \in C(\Omega \times \R \times \R^N \times \text{Sym}(N))$ and $u_k \in C(\overline{\Omega})$ be viscosity subsolutions (resp. supersolutions) of
$G_k(x, u, Du, D^2u) = 0$ in $\Omega$ for $k \in \mathbb{N}$. Assume also that $G_k(x, r, p, X) \to G(x, r, p, X)$ uniformly on compacts subsets of $\Omega \times \R \times \R^N \times \text{Sym}(N)$ and $u_k \to u$ uniformly on $\overline{\Omega}$ for some functions $G$ and $u$, as $k \to \infty$. Then $u$ is a viscosity subsolution (resp. supersolution) of \eqref{EqDiricProb}.
\end{lemma}

The following result states the existence of viscosity solutions via the sub-supersolutions method, cf. \cite[Proposition 2.3]{Ishii89}.

\begin{lemma}[{\bf Existence of solutions via sub-supersolutions}]\label{ExistSol} Let $G \in C(\Omega \times \R \times \R^N \times \text{Sym}(N))$ and suppose that there is a viscosity subsolution $\underline{u}$ and a viscosity supersolution $\overline{u}$ of
\eqref{EqDiricProb}  satisfying  $\underline{u} \leq \overline{u}$ and $\underline{u}, \overline{u} \in C(\overline{\Omega})$. Then  \eqref{EqDiricProb} has a viscosity solution satisfying $\underline{u} \leq u \leq \overline{u}$.
\end{lemma}

In \cite{BD07} Birindelli and Demengel proved the following maximum principle result:

\begin{lemma}[{\bf Maximum Principle}]
Assume $(F1)$ and $(F2)$. Let $u$ be a positive viscosity solution of
$|Du|^\gamma F(x,D^2u) - ku^{1+\gamma}\leq 0$ in  $\Omega$,
for some constant $k$. Suppose that $u(x_0) = 0$ for some $x_0\in\partial\Omega$ such that $\Omega$ satisfies the interior sphere condition at $x_0$. Then there exists $M > 0$ (depending only on the structural data) such that
$u(x)> Md_{\partial\Omega}(x)$,
where $d_{\partial\Omega}$ is the distance to the boundary.
\end{lemma}

As a direct consequence, we are able to prove the following Hopf type lemma:

\begin{lemma}[{\bf Hopf Lemma}]\label{HopfLemma}
Assume $(F1)$ and $(F2)$. Let $u$ be a positive viscosity solution of \\$|Du|^\gamma F(x,D^2u) +a(x)u^{1+\gamma}\leq 0$  in  $\Omega$. If $u(x_0) = 0$ for some $x_0\in\partial\Omega$ such that $\Omega$ satisfies the interior sphere condition at $x_0$, then there exists $M > 0$ (depending only on the structural data) such that
$u(x)> Md_{\partial\Omega}(x)$,
where $d_{\partial\Omega}$ is the distance function to the boundary.
\end{lemma}

\begin{proof}
It is enough to see that
$$
	|Du|^\gamma F(D^2u)-\|a\|_\infty u^{1+\gamma}\leq |Du|^\gamma F(D^2u)+a(x)u^{1+\gamma}\leq 0.
$$
in the viscosity sense. So the previous Lemma yields the conclusion.
\end{proof}

As a byproduct of Lemma \ref{HopfLemma} and the $C^{1,\alpha}$ regularity estimates up to the boundary \cite{BD2} for solutions of the limiting problem of $(P_{a,q})$ with $q=1+\gamma$, namely,
\begin{equation}\tag{{\bf $\mathrm{P}_{a,1+\gamma}$}}
\left\{
\begin{array}{rcrcl}
|Du|^\gamma F(x, D^2u)+a(x)u^{1+\gamma}(x)& =& 0 & \mathrm{in} & \Omega\\
u &\geq& 0 & \mathrm{in} & \Omega\\
u& = & 0 & \mathrm{on} & \partial\Omega,
\end{array}
\right.
\end{equation}
we are able to deduce a  \textit{Hopf Maximum Principle} (or Hopf-Oleinik type result).

\begin{theorem}[{\bf Hopf Maximum Principle}]
Assume $(F1)$ and $(F2)$. Let $u$ be a non-trivial viscosity solution of $(P_{a,1+\gamma})$.
Then $u \in \mathcal{P}^{\circ}(\Omega)$.
\end{theorem}

The next theorem, which plays a fundamental role in proving Theorem \ref{main result 1}, can be found in \cite[Theorem 3.1]{ART15}, \cite[Theorem 1.1]{BD2}, \cite[Theorem 1.1]{BD3}, \cite[Section 8.2]{CC95} and \cite[Theorem 1]{IS}.

\begin{theorem}[{\bf Gradient estimates}]\label{GradThm} Assume $(F1)$, $(F2)$, and $(F3)$. Let $u$ be a bounded viscosity solution of
$$
  |Du|^{\gamma}F(x, D^2u) = f \in L^{\infty}(B_1).
$$
Then, for some universal $\alpha \in (0, 1)$, we have
\begin{equation*}
	  \displaystyle [Du]_{C^{0, \alpha}(B_1)} \leq C\left(N, \gamma, \lambda, \Lambda\right)\left( \|u\|_{L^{\infty}(B_1)}+ \|f\|_{L^{\infty}(B_1)}^{\frac{1}{\gamma+1}}\right).
	\end{equation*}
\end{theorem}

Finally, let us consider an eigenvalue problem related to our class of operators. Given a smooth bounded domain $\mathcal{D} \subset \mathbb{R}^N$, we  denote by $\left(\lambda^+(\mathcal{D}),\phi^+(\mathcal{D})\right)$ the first eigenpair of the eigenvalue problem
\begin{equation}\label{problema de autovalor}
\left\{
\begin{array}{rclcl}
|D\phi|^{\gamma}F(x,D^2\phi) & =&   -\lambda  \phi^{\gamma+1}& \text{in} & \mathcal{D}\\
\phi &  =&   0& \text{on} & \partial \mathcal{D},
\end{array}\right.
\end{equation}
with $\phi^+(\mathcal{D})>0$ and $\|\phi^+(\mathcal{D})\|_\infty =1$. For the existence of this eigenpair, see for instance \cite{BCPR15}, \cite{BD09}, \cite{Dem09},\cite{FV12}, \cite{IkoIshiii12}, \cite{QS08} and references therein.

\section{Existence of non-trivial solutions:  Proof of Theorem \ref{existencia}}\label{SecExist}

Let us establish the existence of non-trivial solutions to \eqref{MEq}.

\begin{proof}[{\bf Proof of Theorem \ref{existencia}}]

First we construct a sub-solution of \eqref{MEq}.
Consider $\bar{\mathcal{B}}\subset \{a>0\}$ a ball and $(\lambda^+(\mathcal{B}),\phi^+(\mathcal{B}))$ the first eigenpair of the  problem (\ref{problema de autovalor}). We claim that for $\epsilon>0$ sufficiently small $\epsilon\phi^+$ is a sub-solution of
$$
\begin{array}
[c]{lll}
|Du|^\gamma F(x, D^2u)+a(x)u^q(x)=0& \text{in} & \mathcal{B}.\\
\end{array}
$$
Indeed, since $a>0$ on $\bar{\mathcal{B}}$  we have
$$
	\lambda^+\epsilon^{1+\gamma} \phi^+=\lambda^+\epsilon^{1+\gamma-q}(\phi^+)^{1-q}(\epsilon\phi^+)^q\leq a(x)(\epsilon\phi^+)^q \mbox{ in }\mathcal{B},
$$
when $\epsilon>0$ is small enough. In addition,
$$
	|D(\epsilon\phi^+)|^\gamma F(x,D^2(\epsilon\phi^+))=-\epsilon^{1+\gamma}\lambda^+\phi^+.
$$
So we can conclude that
$$
|D(\epsilon\phi^+)|^\gamma F(x,D^2(\epsilon\phi^+))\geq -a(x)(\epsilon\phi^+)^q,
$$
which proves the claim.

Now we are going to prove that the function
$$
\underline{u}_\epsilon=\left\{\begin{array}{rll}
\epsilon\phi^+ & \text{in} & \mathcal{B},\\
 0 & \text{in} & \Omega\setminus \mathcal{B}\\
\end{array}
\right.
$$
is a viscosity sub-solution of \eqref{MEq}. Let $\varphi\in C^2(\Omega)$ touch $\underline{u}_\epsilon$ from above at $x_0$. If $\varphi(x_0)=\underline{u}_\epsilon(x_0)=0$ then, since $0$ is a viscosity solution of \eqref{MEq}, we have
$$
	|D\varphi(x_0)|^\gamma F(x_0, D^2\varphi(x_0))\geq 0=-a(x_0)(\varphi(x_0))^q.
$$
Now if $\varphi(x_0)=\underline{u}_\epsilon(x_0)>0$ then $x_0\in \mathcal{B} $, therefore $\varphi$ touches $\epsilon\phi^+$ from above at $x_0$. Therefore, we have
$$
	|D\varphi(x_0)|^\gamma F(x_0, D^2\varphi(x_0))=-a(x_0)\varphi^q(x_0),
$$
which proves the claim.\\

We look now for a viscosity super-solution of \eqref{MEq}.  To this end, consider $\psi>0$  satisfying
$$
\left\{\begin{array}{rclcl}
 |D\psi|^\gamma F(x, D^2\psi)& = & -\|a\|_\infty& \text{in} & \Omega,\\
 \psi & = & 0 & \text{in} &\partial \Omega.\\
\end{array}
\right.
$$
and $k \gg 1$ such that
$$
	k^{1+\gamma-q}>\psi^q,
$$
so that
$$
|D(k\psi)|^\gamma F(x, D^2(k\psi))=-k^{1+\gamma}\|a\|_\infty\leq -\|a\|_\infty(k\psi)^q.
$$
Since $\|a\|_\infty\geq a(x)$ we have
$$
   |D(k\psi)|^\gamma F(x, D^2(k\psi))\leq -a(x)(k\psi)^q.
$$
Thus $\bar{u}_k=k\psi$ is a super-solution of \eqref{MEq}. As $\underline{u}_\epsilon=0$ in a neighborhood of $\partial \Omega$, we see that taking $\epsilon$ smaller and $k$ larger if necessary,  we have $\underline{u}_\epsilon\leq \bar{u}_k$.

By Lemma \ref{ExistSol} we conclude that there exist a viscosity solution $u$ of \eqref{MEq} such that
$$
	\underline{u}_\epsilon\leq u \leq \bar{u}_k,
$$
and therefore $u\geq 0$ is non-trivial, which proves the theorem.
\end{proof}

\section{Uniqueness of positive solutions: Proof of Theorem \ref{unicidade}}\label{SecUnique}

\hspace{0.6cm}We adapt ideas from \cite{CP09} and \cite{IS} to show that \eqref{MEq} admits at most one positive solution. Nevertheless, the advantage of the method applied here is that we allow $a$ to change sign.

\begin{lemma}\label{change of variables} Let $0<q<\gamma+1$ and $u$ be a positive viscosity solution of
\begin{equation} \label{eu}
	|D u|^\gamma F(x, D^2u)+a(x)u^q(x)=0 \quad \text{in} \quad \Omega.
\end{equation}
Then
$$
	w(x)=\frac{1}{1-\frac{q}{1+\gamma}}u(x)^{1-\frac{q}{1+\gamma}}
$$
is a positive viscosity solution of
\begin{equation} \label{ew}
	|D w|^\gamma F\left(x, D^2w+\left(\frac{q}{1+\gamma-q}\right)\frac{D w\otimes D w}{w}\right)+a(x)=0.
\end{equation}
\end{lemma}

\begin{proof}
We set $\bar{q}=\frac{q}{1+\gamma}$. Consider $\Phi\in C^2(\Omega)$ a test function for \eqref{ew} that touches $w$ from below at $x=x_0 \in \Omega$, so that
$$
  \phi(x)=[(1-\bar{q})\Phi(x)]^\frac{1}{1-\bar{q}}
$$
touches $u$ from below at $x=x_0$. Since $u>0$ and $\Phi\in C^2(\Omega)$ we see that $\phi\in C^2$ in a neighborhood of $x_0$, so it is a test function for \eqref{eu}. Besides, we remark that
$$
  D \phi(x)=\left[(1-\bar{q})\Phi(x)\right]^{\frac{\bar{q}}{1-\bar{q}}}D \Phi(x)
$$
and
$$
  D^2\phi(x)=[(1-\bar{q})\Phi(x)]^{\frac{\bar{q}}{1-\bar{q}}}\left(D^2\Phi(x)+\left(\frac{\bar{q}}{1-\bar{q}}\right)\frac{D \Phi(x) \otimes D \Phi(x)}{\Phi(x)}\right).
$$
Hence,
$$
	-a(x_0)u^q(x_0)\geq |D \phi(x_0)|^\gamma F(x_0, D^2\phi(x_0)),
$$
and by the $1$-homogeneity of the operator $F$ we have
$$
	-a(x_0)u^q(x_0)\geq [(1-\bar{q})\Phi(x_0)]^{(1+\gamma)\frac{\bar{q}}{1-\bar{q}}}|D \Phi(x_0)|^\gamma F\left(x_0, D^2\Phi(x_0)+\left(\frac{\bar{q}}{1-\bar{q}}\right)\frac{D \Phi(x_0)\otimes D \Phi(x_0)}{\Phi(x_0)}\right).
$$
Thus, using the definition of $\phi $ we obtain
$$
	-a(x_0)[(1-\bar{q})\Phi(x_0)]^\frac{q}{1-\bar{q}}\geq [(1-\bar{q})\Phi(x_0)]^{\frac{(1+\gamma)\bar{q}}{1-\bar{q}}} |D \Phi(x_0)|^\gamma F\left(x_0, D^2\Phi(x_0)+\left(\frac{\bar{q}}{1-\bar{q}}\right)\frac{D \Phi(x_0) \otimes \Phi(x_0)}{\Phi(x_0)}\right).
$$
Finally, as $(1+\gamma)\bar{q}=q$, we conclude that
$$
	-a(x_0)\geq |D \Phi(x_0)|^\gamma F\left( x_0, D^2\Phi(x_0)+\left(\frac{\bar{q}}{1-\bar{q}}\right)\frac{D \Phi(x_0)\otimes \Phi(x_0)}{\Phi(x_0)}\right),
$$
which proves that $w$ is a viscosity super-solution of \eqref{ew}. In a similar way, one proves that $w$ is a viscosity sub-solution of \eqref{ew}.
\end{proof}

With a little help from Lemma \ref{change of variables} we are able to prove our uniqueness result.

\begin{proof}[{\bf Proof of Theorem \ref{unicidade}}]
Suppose that \eqref{eu} has two positive solutions $u_1, u_2$. By Lemma \ref{change of variables}, it follows that \eqref{ew} has two positive solutions $w_1, w_2$. Let $x_0\in \Omega$ be such that
\begin{equation}\label{hipotese de contradicao unicidade}
	\max_{\bar{\Omega}}(w_1-w_2)=(w_1-w_2)(x_0)>0
\end{equation}
and $\Omega_1 $ be an open connected neighborhood of $x_0$ where $w_1-w_2>0$ . We claim that
$$
	\mathcal{M}_{\lambda, \Lambda}^+(D^2(w_1-w_2))\geq 0 \mbox{ in } \Omega_1.
$$
To this end we consider
$$
  \phi(x)=\frac{1}{2}x^T.A.x+bx+c, \quad (A, b, c) \in \left(\mbox{Sym}(N), \R^N, \R\right)
$$
a second order polynomial that touches $w_1-w_2$ from above at $x_1\in \Omega_1 $, i.e. $\phi(x_1)=(w_1-w_2)(x_1)$ and $\phi> w_1-w_2 \,\,\,\mbox{in}\,\,\, B_r\setminus \{x_1\}\subset \Omega_1$ for some $r\ll 1$.
We are going to show that
\begin{equation}\label{desigualdade importante 1 do teorema de unicidade}
	\mathcal{M}_{\lambda, \Lambda}^+(A)\geq 0.
\end{equation}
We consider two cases:

First, if $b\neq 0$ then $w_1-w_2(x_1)-\phi$ has a maximum point at $x_1$, so

$$
	|b|^\gamma F\left(x_1, A+\left(\frac{q}{1+\gamma-q}\right)\frac{b\otimes b}{w_1(x_1)}\right)\geq-a(x_1).
$$
In a similar way, $w_2-w_1(x)+\phi$ has a minimum point at $x_1$, so that
$$
	|b|^\gamma F\left(x_1, -A+\left(\frac{q}{1+\gamma-q}\right)\frac{b\otimes b}{w_2(x_1)}\right)\leq-a(x_1),
$$
Therefore, we obtain
$$
   |b|^\gamma F\left(x_1, A+\left(\frac{q}{1+\gamma-q}\right)\frac{b\otimes b}{w_1(x_1)}\right)-|b|^\gamma F\left(x_1, -A+\left(\frac{q}{1+\gamma-q}\right)\frac{b\otimes b}{w_2(x_1)}\right)\geq 0.
$$
Now, using the condition (\textbf{F1}) we have
$$
|b|^\gamma \mathcal{M}_{\lambda, \Lambda}^+\left(2A+\left(\frac{q}{1+\gamma-q}\right)\left( \frac{1}{w_1(x_1)}-\frac{1}{w_2(x_1)}\right)b\otimes b\right)\geq 0,
$$
and therefore, since $b\neq 0$, we have
$$
\mathcal{M}_{\lambda, \Lambda}^+\left(2A+\left(\frac{q}{1+\gamma-q}\right)\left( \frac{1}{w_1(x_1)}-\frac{1}{w_2(x_1)}\right)b\otimes b\right)\geq 0.
$$
To simplify the notation, we set
$$
	B=\left(\frac{q}{1+\gamma-q}\right)\left( \frac{1}{w_1(x_1)}-\frac{1}{w_2(x_1)}\right)b\otimes b.
$$
We know that $b\otimes b$ is a non-negative symmetric matrix, and by (\ref{hipotese de contradicao unicidade}),
$$
\frac{1}{w_1(x_1)}-\frac{1}{w_2(x_1)}<0,
$$
so $B$ is a non-positive, symmetric matrix.

On the other hand we know that $\mathcal{M}_{\lambda, \Lambda}^+$ is a $(\lambda, \Lambda)$-elliptic operator, so
$$
	\mathcal{M}_{\lambda, \Lambda}^+(2A)+\Lambda\|B^+\|-\lambda\|B^-\|\geq \mathcal{M}_{\lambda, \Lambda}^+(2A+B)\geq 0.
$$
As $B$ is a non-positive, we have $\|B^+\|=0$, and consequently
$$
\mathcal{M}_{\lambda, \Lambda}^+(2A)\geq \lambda\|B^-\|=\lambda\|-B\|.
$$
Thus  $\mathcal{M}_{\lambda, \Lambda}^+(A)\geq 0$, as claimed.

Assume now that $b=0$ and suppose, for the sake of contradiction, that
\begin{equation}\label{hipotese de contradição para o caso quando b=0}
\mathcal{M}_{\lambda, \Lambda}^+(A)<0,
\end{equation}
so that $A$ has at least one non-positive eigenvalue. Let $S$ be the direct sum of the subspaces associated with the non-positive eigenvalues of $A$ and $P_S$ be the orthogonal projection over $S$. Setting
$$
	\psi(x)=\phi(x)-\epsilon|P_Sx|
$$
and recalling that $\phi> w_1-w_2 \,\,\,\mbox{in}\,\,\, B_r\setminus \{x_1\}$, we have that, for $0<\epsilon \ll 1$ small enough, $\psi$ touches $w_1-w_2$ from above at some $x_2\in \overline{B_r(x_1)}$ such that $(w_1-w_2)(x_2)>0$. This fact suggests us to use $\psi$ as a test function at $x_2$. To this end, let us prove that $|P_Sx_2|\neq 0$. In effect, by supposing the opposite, we use that
\begin{equation}\label{definição da norma da transformação linear}
   |P_Sx|=\max_{|e|=1} P_Sx\cdot e
\end{equation}
 to deduce that for all $e \in \mathbb{S}^{N-1}$, the test function
$$
	\phi(x) -\epsilon P_Sx \cdot e
$$
touches $w_1-w_2$ from above at $x_2$. In fact, we use (\ref{definição da norma da transformação linear}) to assure that
$$
	\phi-\epsilon P_Sx \cdot e\geq \phi- \epsilon|P_Sx|.
$$
We also observe  that
$$
-|P_Sx|\leq P_Sx \cdot e \leq |P_Sx|,
$$
so, by the contradiction assumption,
$$
	P_Sx_2 \cdot e=|P_Sx_2|=0.
$$
Therefore, as
$$
	(w_1-w_2)-(\phi- \epsilon|P_Sx|)
$$
has a local minimum at $x_2$, we conclude that
$$
	(w_1-w_2)-(\phi- \epsilon P_Sx \cdot e)
$$
has also a local minimum at $x_2$, which shows that $\phi- \epsilon P_Sx \cdot e$ touches $w_1-w_2$ from above at $x_2$ for all $e \in \mathbb{S}^{N-1}$.

Arguing as above, we conclude that, for all $e \in \mathbb{S}^{N-1}$, we have
$$
	|Ax_2-\epsilon P_Se|^\gamma F\left(x_2, A+\left(\frac{q}{1+\gamma-q}\right)\frac{Ax_2-\epsilon P_Se\otimes Ax_2-\epsilon P_Se}{w_1(x_2)}\right)\geq-a(x_2).
$$
and
$$
	|Ax_2-\epsilon P_Se|^\gamma F\left(x_2, -A+\left(\frac{q}{1+\gamma-q}\right)\frac{Ax_2-\epsilon P_Se\otimes Ax_2-\epsilon P_Se}{w_2(x_2)}\right)\leq-a(x_2).
$$
So
$$
|Ax_2-\epsilon P_Se|^\gamma \mathcal{M}_{\lambda, \Lambda}^+\left(2A+\left(\frac{q}{1+\gamma-q}\right)\left( \frac{1}{w_1(x_2)}-\frac{1}{w_2(x_2)}\right)(Ax_2-\epsilon P_Se)\otimes (Ax_2-\epsilon P_Se)\right)\geq 0.
$$
Since we can choose $\hat{e} \in \mathbb{S}^{N-1}$ such that $Ax_2-\epsilon P_S\hat{e}\neq 0$, and we know that
$$
	\frac{1}{w_1(x_2)}-\frac{1}{w_2(x_2)}<0,
$$
arguing once more as above, we conclude that
$$
	\mathcal{M}_{\lambda, \Lambda}^+(A)\geq0,
$$
which contradicts (\ref{hipotese de contradição para o caso quando b=0}).  Therefore, $|P_Sx_2|\neq 0.$

Now, we use $\psi$ as a test function to obtain
$$
	|Ax_2-\epsilon e_1|^\gamma \mathcal{M}_{\lambda, \Lambda}^+\left(2(A-\epsilon B)+\left(\frac{q}{1+\gamma-q}\right)\left( \frac{1}{w_1(x_2)}-\frac{1}{w_2(x_2)}\right)(Ax_2-\epsilon e_1)\otimes (Ax_2-\epsilon e_1)\right)\geq 0,
$$
where $e_1=\frac{P_Sx_2}{|P_Sx_2|}$ and $B=D^2(|P_Sx_2|)\geq 0$, since $|P_Sx|$ is a convex function. Furthermore,
$$
	(Ax_2-\epsilon e_1).P_Sx_2= P_SAx_2.x_2-\epsilon|P_Sx_2|\leq -\epsilon|P_Sx_2|<0,
$$
so  $(Ax_2-\epsilon e_1)\neq 0$. Using once again $$
	 \frac{1}{w_1(x_2)}-\frac{1}{w_2(x_2)}<0
$$
and arguing as above we conclude that
$$
	  \mathcal{M}_{\lambda, \Lambda}^+(A-\epsilon B)> 0.
$$
Since  $ \mathcal{M}_{\lambda, \Lambda}^+(A)\geq  \mathcal{M}_{\lambda, \Lambda}^+(A-\epsilon B)$ we have that
$$
	\mathcal{M}_{\lambda, \Lambda}^+(A)\geq 0
$$
which contradicts (\ref{hipotese de contradição para o caso quando b=0}). Thus we  conclude that
$$
	\mathcal{M}_{\lambda, \Lambda}^+(D^2(w_1-w_2))\geq 0 \mbox{ in } \Omega_1.
$$
Now, we observe that $w_1-w_2=0$ on $\partial\Omega_1$, so that by the maximum principle
$$
	w_1-w_2\leq 0 \mbox{ in } \Omega_1,
$$
which is a contradiction with (\ref{hipotese de contradicao unicidade}). Therefore we have $w_1\leq w_2$. The converse inequality is obtained similarly.
\end{proof}

\section{Positivity results: Proof of Theorem \ref{main result 1}}\label{SecProofMR}

We start proving that nontrivial solutions of \eqref{MEq} are positive in some component of $\Omega^+$.

\begin{lemma}
Let $u$ be a nontrivial viscosity solution of \eqref{MEq}. Then, there exists a sub-domain $\Omega^{\prime}\subset \Omega^+$ such that $u>0$ in $\Omega^{\prime}$.
\end{lemma}

\begin{proof} Assume, for the sake of contradiction, that $u=0$ in $\Omega^+$. Then $\mathcal{A}=\left\lbrace a<0\right\rbrace\cap \left\lbrace u>0\right\rbrace\neq \emptyset$, since otherwise we would have
$$
\left\{\begin{array}{rclcl}
|Du|^\gamma F(x, D^2u)   =   0 & \text{in} & \Omega,\\
u   =   0 & \text{on} & \partial\Omega,
\end{array} \right.
$$
which would imply $u\equiv 0$. Now, let $x_0\in \overline{\mathcal{A}}$ be such that $\displaystyle u(x_0)=\max_{\overline{\mathcal{A}}} u$ and take $\varphi(x)\equiv u(x_0)$ as a test function. Then the maximum of $u-\varphi$ is achieved at $x_0$, so that
$$
	|D\varphi (x_0)|^\gamma F(x_0, D^2\varphi(x_0) ) \geq -a(x_0)u(x_0)^q>0.
$$
On the other hand, $|D\varphi|^{\gamma}F(x_0,D^2\varphi)\equiv 0$, and we obtain a contradiction.
\end{proof}

The next lemma plays a fundamental role in proving our main results. It provides the existence of a `` barrier from below'', which ensures that the set of nontrivial solutions of \eqref{MEq} with $q \in (0,\gamma +1)$  stays away from zero as long as $a$ is sufficiently positive in some ball of $\Omega$.

Let $\bar{\mathcal{B}} \subset \Omega^+$ and $(\lambda^+(\mathcal{B}),\phi^+(\mathcal{B}))$ the first eigenpair of the  problem (\ref{problema de autovalor}).

\begin{lemma}\label{l2}
Let $q\in (0, \gamma+1)$. If $\displaystyle a_0:=\min_{\overline{\mathcal{B}}}a>\lambda^+(\mathcal{B})>0$ then there exists $\Phi\in C(\overline{\mathcal{B}})$ such that $u\geq \Phi> 0$ in $\mathcal{B}$, for every $u$ satisfying $u>0$ in $\mathcal{B}$ and $|Du|^\gamma F(x, D^2u)+a(x)u(x)^q=0$ in  $\mathcal{B}$.

\end{lemma}

\begin{proof}
Let $\theta>0$ be such that $a_0(\lambda^+)^{-1}-\theta>1$, where $\lambda^+=\lambda^+(\mathcal{B})$.
We consider $\epsilon\phi$, with $\phi=\phi^+(\mathcal{B})$ and
\begin{equation}\label{eub}
  \displaystyle \epsilon<\epsilon_\theta \defeq \left\{a_0(\lambda^+)^{-1}-\theta\right\}^{\frac{1}{\gamma+1-q}}.
\end{equation}
Note that by our choice of $\theta$ we have $\epsilon_\theta >1$ for every $q \in (0, \gamma+1)$. Moreover,
\begin{equation}\label{1}
	|D(\epsilon\phi)|^{\gamma}F(D^2\epsilon\phi)=-\lambda^+ (\epsilon\phi)^{\gamma+1}> -a(x)(\epsilon\phi)^q+\theta\lambda^+(\epsilon\phi)^q \quad \text{in } \mathcal{B}
\end{equation}
whenever $0<\epsilon\leq \epsilon_\theta$. Let $u$ be a positive solution of $|Du|^\gamma F(x, D^2u)= -a(x)u(x)^q$ in  $\mathcal{B}$. We claim that $u\geq\epsilon_\theta \phi$ in $\mathcal{B}$. Indeed, otherwise we choose $x_0\in \mathcal{B}$ such that $\epsilon_\theta \phi(x_0)>u(x_0)$. Now, since $u>0$ on $\overline{\mathcal{B}}$ and $\phi=0$ on $\partial \mathcal{B}$ we can find some $s_0\in(0,1)$ such that $s_0\epsilon_\theta \phi(x_0)=u(x_0)$ and $s_0\epsilon_\theta \phi\leq u$ in $\mathcal{B}$. Thus, the function
$$
	\Theta(x)\defeq s_0\epsilon_\theta \phi(x)
$$
touches $u$ from below at $x_0$. Consequently, we have, in the viscosity sense,
$$
	|D\Theta(x_0)|^{\gamma}F(D^2\Theta(x_0))\leq  -a(x_0)\Theta(x_0)^q.
$$
On the other hand, by \eqref{1} we have
\begin{eqnarray*}
  |D\Theta(x_0)|^{\gamma}F(D^2\Theta(x_0))&=&|D\Theta(x_0)|^{\gamma} \mathcal{M}_{\lambda, \Lambda}^+(D^2 s_0\epsilon_\theta \phi(x_0))\\
  &>&-a(x_0)(s_0\epsilon_\theta \phi(x_0))^q+\theta\lambda^+s_0\epsilon_\theta \phi(x_0)
\end{eqnarray*}
Moreover, since $\Theta(x_0)=s_0\epsilon_\theta \phi(x_0)$, we obtain
$$
	-a(x_0)(s_0\epsilon_\theta \phi(x_0))^q>-a(x_0)(s_0\epsilon_\theta \phi(x_0))^q+\theta\lambda^+s_0\epsilon_\theta \phi(x_0),
$$
which provides a contradiction. Therefore  $u \geq \Phi := \epsilon_\theta \phi$ in $\mathcal{B}$.
\end{proof}

\begin{remark}
The previous lemma holds under the condition
\begin{equation}\label{condicao do infa0}
\displaystyle a_0:=\min_{\overline{\mathcal{B}}}a>\lambda^+(\mathcal{B})>0.
\end{equation}
Let us show that, in the proof of Theorem \ref{main result 1}, we can always assume such hypothesis. In fact, since $F$ is an $1$-homogeneous operator, we see that given $c>0$, $u$ solves $(P_{a,q})$ if and only if
$$
\tilde{u}(x) \defeq c^{\frac{1}{1+\gamma-q}}u(x)
$$
solves $(\mathrm{P}_{ca,q})$. Hence, for $\mathfrak{c}$ sufficiently large, we have
$$
\displaystyle a_0:=\min_{\overline{\mathcal{B}}}\mathfrak{c}a> \lambda^{+}(\mathcal{B}).
$$
Therefore (\ref{condicao do infa0}) holds for $(\mathrm{P}_{ca,q})$, i.e. for $ca$ with $c$ large enough. Since $u$ is positive if and only if $\tilde{u}$ is positive, we see that the above condition can be assumed without loss of generality.
\end{remark}

\begin{proof}[{\bf Proof of Theorem \ref{main result 1}}]
\strut
\begin{enumerate}
\item Suppose, by contradiction, that there exists a sequence $a_k \in C(\overline{\Omega})$ with
$$
	\displaystyle a_k^{+}=a^{+} \quad \text{and} \quad a_k^- \to 0 \text{ in } C(\overline{\Omega}),
$$
and $(P_{a_k,q})$ has viscosity solutions $u_k$ such that $u_k \not \in \mathcal{P}^{\circ}$ for every $k$. Since $\Omega_+$ has finitely many connected components, we can assume that $u_k >0$ in some sub-domain $\Omega^{\prime} \subset \Omega$ for every $k$.
We  split the analysis in two cases;

\begin{enumerate}
  \item First suppose that $\|u_k\|_{L^{\infty}(\Omega)}< C$, for all $k$.
In this case, from $C^{1, \alpha}$ regularity estimates (see \cite[Theorem 3.1]{ART15} \cite{BD2},\cite[Theorem 1]{IS}), we obtain that
$$
[u_k]_{C^{1, \alpha}(\Omega)}\leq C(N, \lambda, \Lambda, \gamma, \Omega)\left[\|u_k\|_{L^{\infty}(\Omega)}+ \|a_ku_k^q\|^{\frac{1}{\gamma+1}}_{L^{\infty}(\Omega)}\right]
$$
and
$$
\begin{array}{cl}
  u_k \to u_0 & \text{locally uniformly in} \,\,\, C^{1, \alpha}(\Omega), \\
  F_k \to F_0 &  \text{locally uniformly in} \,\,\, \text{Sym}(N),\\
  a_k^{-} \to 0 & \text{locally uniformly in} \,\,\, \Omega.
\end{array}
$$
Thus, we have in the viscosity sense
$$
   |D u_0(x)|^{\gamma}F_0(x, D^2 u_0)  +a^{+}(x)u_0^{q}(x)=0  \quad \text{in} \quad \Omega
$$
Furthermore, by the Strong Maximum Principle we infer that either $u_0 \in \mathcal{P}^{\circ}(\Omega)$ or
$u_0\equiv 0$.
However, by Lemma \ref{l2}, we know that there exists $\Phi$ such that, for every $k \in \mathbb{N}$,
$$
   u_k \geq \Phi>0 \quad \text{in} \quad \mathcal{B} \subset supp(a^{+}).
$$
Hence, we conclude that $u_0 \in \mathcal{P}^{\circ}(\Omega)$. Thus, since $u_k \to u_0$ locally uniformly in $C^{1, \alpha}(\Omega)$ we see that  $u_k\in \mathcal{P}^{\circ}(\Omega)$ for $k$ sufficiently large,  which provides a contradiction and completes the proof in this case.\\

  \item Suppose now that $\|u_k\|_{L^{\infty}(\Omega)}\rightarrow \infty$ as $k \to \infty$.
We set $v_k(x)=\frac{u_k(x)}{\|u_k\|_{L^{\infty}(\Omega)}}$, so that
$$
Dv_k=\frac{Du_k}{\|u_k\|_{L^{\infty}(\Omega)}}. 
$$
By the $1$-Homogeneity condition (see \ref{F2}), we have
$$
 |Dv_k|^\gamma F(x,D^2v_k)=\frac{1}{\|u_k\|_{L^{\infty}(\Omega)}^{1+\gamma}} |Du_k|^\gamma F(x,D^2u_k)=-\frac{a(x)u_k^q}{\|u_k\|_\infty^{1+\gamma}}=-\frac{a(x)v_k^q}{\|u_k\|_{L^{\infty}(\Omega)}^{1+\gamma-q}}.
$$
in the viscosity sense. By arguing as above, we conclude that
$$
\begin{array}{cl}
  v_k \to v_0 & \text{locally uniformly in} \,\,\, C^{1, \alpha}(\Omega) \\
\end{array}
$$
Moreover,  $\|u_k\|_{L^{\infty}(\Omega)}^{1+\gamma-q}\rightarrow \infty$ and $a(x)v_k^q$ remains bounded. Therefore,
$$
	|Dv_0|^\gamma F(x,D^2v_0)=0.
$$
From the Strong Maximum Principle, we infer either $v_0 \in \mathcal{P}^{\circ}(\Omega)$ or $v_0\equiv 0$. Since  $\|v_k\|_{L^{\infty}(\Omega)}=1$ we rule out the second possibility, and therefore
$v_0 \in \mathcal{P}^{\circ}(\Omega)$.
As $v_k \to v_0$ locally uniformly in $C^{1, \alpha}(\Omega)$, for $k$ sufficiently large, we infer again that $v_k \in \mathcal{P}^{\circ}(\Omega)$ and thus $u_k \in \mathcal{P}^{\circ}(\Omega)$, which yields another contradiction.
\end{enumerate}

\item Let us assume, for the sake of contradiction, that there are sequences $\{q_k\} \subset \R_{+}$ such that
$q_k \to (\gamma+ 1)^-$
and $\{u_k\}_{k \in \mathbb{N}}$ are nontrivial viscosity solutions of $(\mathrm{P}_{a,q_k})$ with $u_k \not \in \mathcal{P}^{\circ}$ for every $k$. Since $\Omega_+$ has finitely many connected components, we can assume that $u_k>0$ in some sub-domain $\Omega^{\prime} \subset \Omega$ for every $k$.

As in the previous item, we consider two cases:

\begin{enumerate}
  \item  First we assume that $\{u_k\}_{k \in \mathbb{N}}$ is bounded in $L^{\infty}(\Omega)$.
 Then, as in the previous argument, $\{u_k\}_{k \in \mathbb{N}} $ is bounded in $C^{1,\alpha}(\overline{\Omega})$ for some $\alpha \in (0,1)$. By the Arzel\`{a}-Ascoli compactness criteria, up to a subsequence, we obtain that
$$
  u_k\rightarrow u_0 \quad \text{in} \quad C^{1,\alpha}(\overline{\Omega}).
$$
By Lemma \ref{l2}, we know that
$$
u_k>C \quad \text{in} \quad \mathcal{B} \subset supp(a^{+}).
$$
so that
$$\
u_0>C \quad \text{in} \quad \mathcal{B} \subset supp(a^{+}).
$$
for some constant $C>0$, i.e. $u_0 \not \equiv 0$. Moreover, by stability, $u_0$ solves $(\mathrm{P}_{a,1+\gamma})$. The Hopf Maximum Principle, yields that $u_0 \in \mathcal{P}^{\circ}(\Omega)$, and since $u_k\rightarrow u_0$ in $C^{1,\alpha}(\overline{\Omega})$, it follows that $u_k\in \mathcal{P}^{\circ}(\Omega)$ for $k$ large enough, which yields a contradiction.\\

  \item Now, suppose that $\|u_k\|_\infty\rightarrow\infty$ and set
$$
v_k=\frac{u_k}{\|u_k\|_\infty},
$$
which solves
$$|D v_k|^{\gamma}F(x, D^2v_k)=-\frac{a(x)v_k^{q_k}}{\|u_k\|_\infty^{(\gamma+1)-q_k}} \text{ in } \Omega, \quad v_k=0 \text{ on } \partial \Omega.
$$
By the Arzel\`{a}-Ascoli theorem, we have, up to a subsequence,
$v_k\rightarrow v_0
$ in $C^{1,\alpha}(\overline{\Omega})$.

Once again, we must analyse two possibilities:

\begin{itemize}
  \item  If $\|u_k\|_\infty^{(\gamma+1)-q_k}\rightarrow +\infty$
then  $v_0$ solves
$$
	|D v_0|^{\gamma}F(x, D^2v_0)=0, \text{ in } \Omega, \quad v_0=0 \text{ on } \partial \Omega.
$$
Since $v_0 \not \equiv 0$, we deduce, again by the Hopf Maximum Principle, that $v_0 \in \mathcal{P}^{\circ}(\Omega)$, so that $v_k \in \mathcal{P}^{\circ}(\Omega)$ for $k$ large enough, which yields a contradiction.
  \item If $\|u_k\|_\infty^{(\gamma+1)-q_k}$ is bounded
then
$$
	\|u_k\|_\infty^{(\gamma+1-q_k)}\rightarrow C_0.
$$
Note that $C_0\geq 1$ since $\|u_k\|_\infty>>1$, so $v_0$ solves
$$
|D v_0|^{\gamma}F(x, D^2v_0)=-\frac{a(x)v_0^{\gamma+1}}{C_0}, \text{ in } \Omega, \quad v_0=0 \text{ on } \partial \Omega.
$$
As $v_0 \not \equiv 0$, by the Hopf Maximum Principle, $v_0 \in \mathcal{P}^{\circ}(\Omega)$, so that $v_k \in \mathcal{P}^{\circ}(\Omega)$ for $k$ large enough, which yields another contradiction. The proof is now complete.
\end{itemize}
\end{enumerate}
\end{enumerate}
\end{proof}

To conclude this paper we provide a simple example of a problem having a dead-core solution:

\begin{example} Let $0\leq q < \gamma+1$ and $0\leq \gamma <2q$. Consider $v: [0, \pi] \to \R $ given by
$$
v(x) = \frac{\sin^r(x)}{r}
$$
where $r \defeq \frac{\gamma+2}{\gamma+1-q}$. Then, it is easy to check that

$$
	v'(x)=\sin^{r-1}x.\cos x \mbox{ and } v''(x)=(r\cos^2x-1)\sin^{r-2}x.
$$
As $r>2$ we have that $v\in C^2([0,\pi])$, so we can extend $v$ smoothly like $0$ in $(-\pi/2,0]$.
Therefore
$$
  -|v^{\prime}(x)|^{\gamma}v^{\prime \prime}(x) = a(x)v^q(x) \quad \text{in} \quad (-\pi/2, \pi)
$$
in the viscosity sense, where
$$
   a(x) = r^q|\cos(x)|^{\gamma}\left[-r\cos^2(x)+1\right]
$$
is a sign-changing weight. Clearly $v$ is a solution of \eqref{MEq} with dead core. Thus
$$
	\delta_0< |a(0)|=r^q(r-1).
$$
\end{example}


\subsection*{Acknowledgments}
This work was partially supported by Consejo Nacional de Investigaciones Cient\'{i}ficas y T\'{e}cnicas (CONICET-Argentina) and Coordena\c{c}\~{a}o de Aperfei\c{c}oamento de Pessoal de N\'{i}vel Superior (PNPD-CAPES-UnB-Brazil). J.V. da Silva would like to thank the Universidad de Buenos Aires for providing an excellent working environment during his stay there, and, Universidade Federal de Sergipe by the best scientific atmosphere during his visit where part of this manuscript was written.

\end{document}